\documentclass[10pt,letterpaper]{article}
\usepackage[utf8]{inputenc}
\usepackage{amsmath}
\usepackage{amsfonts}
\usepackage{amssymb}
\usepackage{graphicx}
\usepackage{mathrsfs}
\usepackage{upref,amsthm,amsxtra,exscale}
\usepackage{cite}
\usepackage[colorlinks=true,urlcolor=blue,
citecolor=red,linkcolor=blue,linktocpage,pdfpagelabels,
bookmarksnumbered,bookmarksopen]{hyperref}
\usepackage{upgreek}
\usepackage{cleveref}
\usepackage[center]{caption}
\usepackage{soul}

\usepackage{fullpage}

\newtheorem{theorem}{Theorem}[section]

\newtheorem{lemma}[theorem]{Lemma}

\newtheorem{problem}[theorem]{Problem}
\newtheorem{remark}[theorem]{Remark}
\numberwithin{equation}{section}

\def\r{\mathbb{R}}

\def\rn{\mathbb{R}^N}
\def\z{\mathbb{Z}}
\def\zn{\mathbb{Z}^N}

\def\eps{\varepsilon}
\def\rh{\rightharpoonup}
\def\io{\int_{\Omega}}
\def\irn{\int_{\r^N}}
\def\vp{\varphi}

\def\o{\Omega}
\def\t{\Theta}

\def\tilde{\widetilde}

\def\cC{\mathcal{C}}

\def\cM{\mathcal{M}}
\def\cN{\mathcal{N}}

\def\supp{\mathrm{supp}}

\def\dist{\mathrm{dist}}
\def\what{\widehat}

\author{Mónica Clapp, Alberto Saldaña and Andrzej Szulkin}
\title{A concentration phenomenon for a semilinear Schrödinger equation with periodic self-focusing core}
\date{}

\begin{document}
\maketitle
	
\begin{abstract}
We consider the equation
$$-\Delta u+u=Q_\eps(x)|u|^{p-2}u,\qquad u\in H^1(\rn),$$
where $Q_\eps$ takes the value $1$ on each ball $B_\eps(y)$, $y\in\z^N$, and the value $-1$ elsewhere. We establish the existence of a least energy solution for each $\eps\in(0,\frac{1}{2})$ and show that their $H^1$ and $L^p$ norms concentrate locally at  points of $\z^N$ as $\eps\to 0$.
\medskip

\medskip

\end{abstract}

\section{Introduction}

Let $Q_\eps:\rn\to\r$ be the function given by
\begin{equation*}
Q_\eps(x):=
\begin{cases}
1 & \text{if \ }x\in\t_\eps:=\bigcup\limits_{y\in\z^N}B_\eps(y),\\
-1 & \text{otherwise},
\end{cases}
\end{equation*}
where $B_r(x)$ denotes the ball of radius $r$ and center at $x$ in $\rn$. Consider the problem
\begin{equation} \label{eq:problem}
\begin{cases}
-\Delta u+u=Q_\eps(x)|u|^{p-2}u,\\
u\in H^1(\rn),
\end{cases} 
\end{equation}
where $\eps\in(0,\frac{1}{2})$ and  $p\in (2,2^*)$, $2^*:=\frac{2N}{N-2}$ is the critical Sobolev exponent if $N\geq 3$ and $2^*:=\infty$ if $N=2$.

This type of equations is present in some models of optical waveguides propagating through a stratified dielectric medium (see \cite{as,s1,s2}).

If instead of $\t_\eps$ we consider a single ball $B_\eps(0)$, the problem has been studied by Ackermann and Szulkin in \cite{as}. They established the existence of a least energy solution $u_\eps$ for each $\eps>0$ and showed that the $H^1$ and $L^p$ norms of $u_\eps$ are concentrated at the origin when $\eps\to 0$. More precisely, they proved that, for any $\delta>0$,
\begin{equation} \label{eq:concentration_as}
\lim_{\eps\to 0}\frac{\int_{B_\delta(0)}(|\nabla u_\eps|^2+u_\eps^2)}{\irn(|\nabla u_\eps|^2+u_\eps^2)}=1\qquad\text{and}\qquad\lim_{\eps\to 0}\frac{\int_{B_\delta(0)}|u_\eps|^p}{\irn |u_\eps|^p}=1.
\end{equation}
Fang and Wang \cite{fw} showed that, after rescaling, the solutions exhibit a limiting profile, which is a least energy solution of a limiting equation. The concentration behavior described above can be easily derived from this fact. Some properties of the limiting profile, such as its symmetries and  decay at infinity, were shown in \cite{chs}.

Following this program for the unbounded set of balls $\t_\eps$ presents serious challenges. Existence can be easily demonstrated.

\begin{theorem} \label{thm:main_existence}
For each $\eps\in(0,\frac12)$ the problem \eqref{eq:problem} has a positive least energy solution.
\end{theorem} 

Concentration is more difficult to understand. We prove the following local result.

\begin{theorem} \label{thm:main_concentration}
Let $\eps_n\to 0$ and, for each $n$, let $u_n$ be a positive solution to \eqref{eq:problem} with $\eps=\eps_n$. Then, after suitable translations by elements of $\zn$, for any open bounded subset $U$ of $\rn$ such that $U\supset[-\frac{1}{2},\frac{1}{2}]^N$ and $\partial U\cap\z^N=\emptyset$, and any $0<\delta<\min\{\frac{1}{2},\,\dist(U\cap\z^N,\partial U)\}$, we have 
\begin{equation} \label{eq:concentration}
\lim_{n\to\infty}\frac{\sum_{y\in U\cap\z^N} \int_{B_\delta(y)}(|\nabla u_n|^2+u_n^2)}{\int_U(|\nabla u_n|^2+u_n^2)}=1\qquad\text{and}\qquad\lim_{n\to\infty}\frac{\sum_{y\in U\cap\z^N}\int_{B_\delta(y)}|u_n|^p}{\int_U |u_n|^p}=1.
\end{equation}
Moreover, $\|u_n\|_{H^1(U)}\to \infty$ as $n\to\infty$.
\end{theorem}

Ackermann and Szulkin showed that, for a finite number of balls, concentration occurs (globally) at a single point \cite[Theorem 4.1]{as}. It is unclear whether this is still true in our case. It is also unclear whether a limiting profile exists.

We would also like to mention some other papers where the set $\{x\in\rn: Q_\eps(x)>0\}$ is assumed to shrink to a point or a finite number of points as $\eps\to 0$. In \cite{liu} the case of a single ball and $p\in(1,2)$ is considered. \cite{sy} is concerned with a limiting profile of solutions for a Schrödinger-Poisson system with $\{x\in\rn: Q_\eps(x)>0\}$ shrinking to one or two points. It also extends one of the results contained in \cite{fw}. In \cite{css}, \cite{jsx} and \cite{zz} systems of two or more equations are considered under different sets of assumptions. 

The paper is organized as follows. Theorem \ref{thm:main_existence} is proved in Section \ref{sec:existence} and Theorem \ref{thm:main_concentration}  in Section \ref{sec:concentration}. Finally, in Section \ref{op} we discuss some open problems.

\section{Existence of a positive least energy solution}\label{sec:existence}

Recall that $N\geq 2$, $\eps\in(0,\frac{1}{2})$ and  $p\in (2,2^*)$. We denote the inner product and the norm in $H^1(\rn)$ by
\begin{equation}\label{eq:inner product}
\langle u,v\rangle:=\irn(\nabla u\cdot\nabla v + uv)\qquad\text{and}\qquad \|u\|^2:=\irn(|\nabla u|^2 + u^2).
\end{equation}
Fix $\eps\in(0,\frac12)$. The solutions of \eqref{eq:problem} are the critical points of the functional $J_\eps:H^1(\rn)\to\r$ defined by
$$J_\eps(u):=\frac{1}{2}\irn(|\nabla u|^2 +  u^2) - \frac{1}{p}\irn Q_\eps|u|^p.$$
$J_\eps$ is of class $\cC^2$ and its derivative at $u$ is given by
$$
J'_\eps(u)v=\irn(\nabla u\cdot\nabla v + uv) - \irn Q_\eps|u|^{p-2}uv , \qquad v\in H^1(\rn).
$$
The nontrivial critical points of $J_\eps$ belong to the Nehari manifold
$$\cN_\eps:=\{u\in H^1(\rn):u\neq 0, \ J'_\eps(u)u=0\},$$
which is a Hilbert submanifold of $H^1(\rn)$ of class $\cC^2$ and a natural constraint for $J_\eps$. Note that
\begin{equation} \label{eq:energy_on_nehari}
J_\eps(u)=\frac{p-2}{2p}\|u\|^2=\frac{p-2}{2p}\irn Q_\eps|u|^p\qquad\text{if \ }u\in\cN_\eps.
\end{equation}
By Sobolev's inequality there is $C>0$, depending only on $N$ and $p$, such that 
$$
\|u\|^2=\irn Q_\eps|u|^p\leq\int_{\t_\eps}|u|^p\leq\irn|u|^p\leq C\|u\|^p\qquad\text{for every \ }u\in\cN_\eps
$$
(recall that $\t_\eps:=\bigcup\limits_{y\in\z^N}B_\eps(y)$). Hence, there exists $C_0>0$, depending only on $N$ and $p$, such that
\begin{align}\label{C0}
c_\eps:=\inf_{u\in\cN_\eps}J_\eps(u)\geq C_0>0\qquad\text{for every \ }\eps\in(0,\textstyle\frac12).
\end{align}

\begin{theorem}\label{thm:existence of minimizer}
Let $u_k\in \cN_\eps$ be such that $J_\eps(u_k)\to c_\eps$. Then there exist $y_k\in\z^N$ such that, after passing to a subsequence, the sequence $(\what u_n)$ given by $\what u_k(x):=u_k(x+y_k)$ converges strongly in $H^1(\rn)$ to a solution $\what u$ of \eqref{eq:problem} that satisfies $J_\eps(\what u)=c_\eps$.
\end{theorem}

\begin{proof}
It follows immediately from \eqref{eq:energy_on_nehari} that $(u_k)$ is bounded. Using Ekeland's variational principle \cite[Theorem 8.5]{w} we may assume that $J_\eps'(u_k)\to 0$ in $H^1(\rn)'$. Furthermore, by \eqref{C0},
$$\frac{2p}{p-2}C_0\leq \|u_k\|^2=\irn Q_\eps|u_k|^p\leq\int_{\t_\eps}|u_k|^p\leq\irn|u_k|^p.$$
Hence $u_k\not\to 0$ in $L^p(\rn)$ and it follows from Lions' lemma \cite[Lemma 1.21]{w} that there exist $\alpha>0$ and $\xi_k\in\rn$ such that
$$\int_{B_1(\xi_k)}|u_k|^2=\sup_{x\in\rn}\int_{B_1(x)}|u_k|^2\geq \alpha>0.$$
Fix $r>\frac{\sqrt{N}}{2}$ and choose $y_k\in\z^N$ such that $|y_k-\xi_k|<r$. Then, 
$$\int_{B_{r+1}(y_k)}|u_k|^2\geq\int_{B_1(\xi_k)}|u_k|^2\geq \alpha>0.$$
The function $\what u_k(x):=u_k(x+y_k)$ satisfies $\|\what u_k\|=\|u_k\|$ and 
\begin{equation}\label{eq:not_zero}
\int_{B_{r+1}(0)}|\what u_k|^2\geq \alpha>0.
\end{equation}
So, passing to a subsequence, $\what u_k\rh\what u$ weakly in  $H^1(\rn)$, $\what u_k\to\what u$ in $L^p_\mathrm{loc}(\rn)$ and a.e. in $\rn$. Let $\vp\in\cC^\infty_c(\rn)$ and set $\vp_k(x):=\vp(x-y_k)$. As
$$o(1)=J_\eps'(u_k)\vp_k=J_\eps'(\what u_k)\vp=\langle\what u_k,\vp\rangle  - \irn Q_\eps|\what u_k|^{p-2}\what u_k\vp,$$
we see passing to the limit that $\what u$ solves \eqref{eq:problem}. It follows from \eqref{eq:not_zero} that $\what u\neq 0$. Therefore, $\what u\in\cN_\eps$ and
$$c_\eps\leq\frac{p-2}{2p}\|\what u\|^2\leq\lim_{k\to\infty}\frac{p-2}{2p}\|\what u_k\|^2=c_\eps.$$
As a consequence, $\what u_k\to\what u$ strongly in $H^1(\rn)$ and $J_\eps(\what u)=c_\eps$, as claimed.
\end{proof}
\medskip

\begin{proof}[Proof of Theorem~\ref{thm:main_existence}]
Theorem \ref{thm:existence of minimizer} establishes the existence of a least energy solution $\what u$ to the problem \eqref{eq:problem}. Then, $u:=|\what u|$ is a non-negative least energy solution. A standard regularity argument shows that $u\in \cC^1(\rn)$. Applying the strong maximum principle \cite[Theorem 1.28]{dp} to the equation
$$-\Delta u+(1-Q_\eps^-|u|^{p-2})u=Q_\eps^+|u|^{p-2}u\geq 0,$$
where $Q_\eps^+:=\max\{Q_\eps,0\}$ and $Q_\eps^-:=\min\{Q_\eps,0\}$, we see that $u>0$ in $\rn$.
\end{proof}

\section{Concentration of positive solutions}\label{sec:concentration}

Fix a connected open subset $\o$ of $\rn$ such that $[-\frac{1}{2},\frac{1}{2}]^N\subset\o\subset(-\frac{3}{4},\frac{3}{4})^N$. Then, 
$$\dist(\o,\,\z^N\smallsetminus\{0\})\geq \frac{1}{4}\qquad\text{and}\qquad\rn=\bigcup\limits_{y\in\z^N}(y+\o).$$
Let $\eps_n\to 0$ and, for each $n$, let $u_n$ be a positive solution to \eqref{eq:problem} with $\eps=\eps_n$. After a suitable translation by elements of $\zn$ we may assume that
\begin{equation}\label{eq:translation}
\|u_n\|^2_{H^1(\o)}=\io(|\nabla u_n|^2 + u_n^2)\geq\int_{y+\o}(|\nabla u_n|^2 + u_n^2)=\|u_n\|^2_{H^1(y+\o)}\qquad\text{for all \ }y\in\z^N.
\end{equation}
To study the behavior of $\|u_n\|^2_{H^1(\o)}$ we use the following lemma.

\begin{lemma} \label{lem:vanishing}
Let
$w\in H^1_\mathrm{loc}(\rn)$ be such that $w\geq 0$ and
$$\irn (\nabla w\cdot\nabla\vp + w\vp)=-C\irn w^{p-1}\vp\qquad\text{for every \ }\vp\in\cC^\infty_c(\rn)$$
and some $C>0$. If there exists $r>0$ such that
$$M:=\sup_{x\in\rn}\int_{B_{2r}(x)}|w|^p<\infty,$$
then $w=0$.
\end{lemma}

\begin{proof}
Let $w$ be as in the statement. Since $w$ is a weak subsolution of $\Delta u-u=0$, by \cite[Theorem 8.17]{gt} there exists $C>0$ depending at most on $N, r$ and $p$ such that
\begin{align}\label{ubd}
w(x)\leq\sup_{B_r(x)}w \le Cr^{-N/p}|w|_{L^p(B_{2r}(x))} \le Cr^{-N/p}M\qquad \text{for all \ }x\in \rn.
\end{align}
Therefore, $w\in L^\infty(\rn)$ and, by interior elliptic regularity theory (see, for instance, \cite[Section 6.3.1]{ev} and  \cite[Theorem 8.24 and Corollary 6.3]{gt}), we also deduce that $w$ is a classical solution and that $\nabla w$ is uniformly bounded in $\rn$.

Let $\psi(x):=e^{-|x|}$ for $x\in \rn$. A direct computation shows that
\begin{align*}
-\Delta \psi(x)+\psi(x) =  \frac{N-1}{|x|} \psi(x)\qquad \text{for \ }x\in \rn\smallsetminus\{0\}.
\end{align*}
Then, integrating by parts and using that $w$ is a classical solution, $N\geq 2$, and $\nabla w$ is uniformly bounded in $\rn$, we obtain
\begin{align*}
0\geq -C\int_{\rn}w^{p-1}\psi
=\int_{\rn} (-\Delta w + w)\psi
=\irn (-\Delta \psi + \psi)w
= (N-1)\int_{\rn} \frac{\psi(x) w(x)}{|x|} \, dx\geq 0,
\end{align*}
but this is only possible if $w=0$ in $\rn$, as claimed.
\end{proof}

\begin{lemma} \label{lem:unbounded}
$\|u_n\|_{H^1(\o)}\to\infty.$
\end{lemma}

\begin{proof}
Arguing by contradiction, suppose $\|u_n\|_{H^1(\o)}\leq C_1$ for some $C_1>0$. Then, passing to a subsequence, $\|u_n\|_{H^1(\o)}\to C_2\geq 0$.

Set $w_n:=u_n/\|u_n\|_{H^1(\o)}$. It follows from \eqref{eq:translation} that
\begin{equation}\label{eq:bound for w_n}
\|w_n\|_{H^1(y+\o)}=\frac{\|u_n\|_{H^1(y+\o)}}{\|u_n\|_{H^1(\o)}}\leq 1\qquad\text{for every \ }y\in\z^N.
\end{equation}
Then,  $w_n\rh w$ weakly in $H^1(\o)$, $w_n\to w$ strongly in $L^p(\o)$ and a.e. in $\o$. As $\rn=\bigcup_{y\in\z^N}(y+\o)$, using \eqref{eq:bound for w_n}, passing to subsequences and employing the diagonal argument, we see that $w$ extends to a function $w\in H^1_\mathrm{loc}(\rn)$ and $w_n\rh w$ weakly in $H^1_\mathrm{loc}(\rn)$, $w_n\to w$ strongly in $L^p_\mathrm{loc}(\rn)$ and a.e. in $\rn$. Note in particular that $w\ge 0$.

Let $\vp\in \cC_c^\infty(\o)$. Then $\supp(\vp)\subset (y_1+\o)\cup\cdots\cup(y_m+\o)$ for some $y_1,\ldots,y_m\in\z^N$ and taking a $\cC^\infty$-partition of unity $\{\pi_i\}$ subordinated to this covering we see that
$$\lim_{n\to\infty}\langle w_n,\vp\rangle=\lim_{n\to\infty}\sum_{i=1}^m\langle w_n,\pi_i\vp\rangle=\sum_{i=1}^m\langle w,\pi_i\vp\rangle=\langle w,\vp\rangle.$$
As $w_n\to w$ in $L^p_\mathrm{loc}(\rn)$ and $u_n$ solves \eqref{eq:problem} with $\eps=\eps_n$,
$$\langle w_n,\vp\rangle=\|u_n\|_{H^1(\o)}^{p-2}\irn Q_{\eps_n}|w_n|^{p-2}w_n\vp\qquad\text{for every \ }\vp\in\cC^\infty_c(\rn),$$
and passing to the limit we obtain
\begin{equation}\label{eq:1}
\langle w,\vp\rangle=-C_2^{p-2}\irn w^{p-1}\vp\qquad\text{for every \ }\vp\in\cC^\infty_c(\rn).
\end{equation}
Fix $r>0$ small enough so that $B_{2r}(x)$ is contained in the union of at most $2^N$ sets of the form $y+\o$, $y\in\z^N$, for each $x\in\rn$. Then, for some positive constant $C$ that depends only on $\o$ and $p$, using \eqref{eq:bound for w_n} and the weak convergence $w_n\rh w$ in $H^1(y+\o)$, we get
\begin{equation}\label{eq:2}
\int_{B_{2r}(x)}|w|^p\leq\sum_{i=1}^{2^N}\int_{y_i+\o}|w|^p\leq C\sum_{i=1}^{2^N}\|w\|^p_{H^1(y_i+\o)}\leq C\,2^N<\infty\qquad\text{for all \ }x\in\rn.
\end{equation}
Statements \eqref{eq:1} and \eqref{eq:2} and Lemma \ref{lem:vanishing} yield  $w=0$.

Let $\chi\in\cC^\infty_c(\rn)$ be such that $0\leq\chi\leq 1$, $\chi(x)=1$ for $x\in\o$ and $\chi(x) = 0$ for $x\in \rn\smallsetminus U$, where $U$ is an open bounded subset of $\rn$ containing $\overline\o$ and such that $\dist(U,\,\z^N\smallsetminus\{0\})\geq \frac{1}{8}$. Then, as $u_n$ solves \eqref{eq:problem} with $\eps=\eps_n$,
\begin{equation}\label{eq:chi0}
\irn u_n\nabla u_n\cdot\nabla\chi+\irn\chi(|\nabla u_n|^2+u_n^2)=\langle u_n,\chi u_n\rangle=\irn\chi Q_{\eps_n}|u_n|^p,
\end{equation}
and for $\eps_n<\frac{1}{8}$ we get
\begin{equation}\label{eq:chi}
\io (|\nabla u_n|^2+u_n^2)-\io Q_{\eps_n}|u_n|^p\leq \irn\chi(|\nabla u_n|^2+u_n^2)-\irn\chi Q_{\eps_n}|u_n|^p= -\irn u_n\nabla u_n\cdot\nabla\chi.
\end{equation}
Dividing this inequality by $\|u_n\|_{H^1(\o)}^2$ gives
\begin{align*}
\io (|\nabla w_n|^2+w_n^2)-\|u_n\|_{H^1(\o)}^{p-2}\io Q_{\eps_n}|w_n|^p\leq C_0\int_U |w_n||\nabla w_n|
\end{align*}
for some $C_0>0$. Since $U$ is covered by finitely many subsets of the form $y+\o$ with $y\in\z^N$ and $\|w_n\|_{H^1(y+\o)}\leq 1$, we have that $\|w_n\|_{H^1(U)}\leq \overline C$ for all $n$ and some constant $\overline C$. Then, recalling that $\|u_n\|_{H^1(\o)}\leq C_1$ and $w_n\to 0$ in $L^p_\mathrm{loc}(\rn)$, we obtain
$$1=\lim_{n\to\infty}\left(\io (|\nabla w_n|^2+w_n^2)-\|u_n\|_{H^1(\o)}^{p-2}\io Q_{\eps_n}|w_n|^p\right)\leq\lim_{n\to\infty} C_0\int_U |w_n||\nabla w_n|=0.$$
This is a contradiction.
\end{proof}
\medskip

Recall that
$$w_n:=\frac{u_n}{\|u_n\|_{H^1(\o)}}.$$

\begin{lemma} \label{weaktozero}
After passing to a subsequence, $w_n\rh 0$ weakly in $H^1(\rn)$, $w_n\to 0$ in $L^p_\mathrm{loc}(\rn)$ and a.e. in $\rn$.
\end{lemma}

\begin{proof}
By the same argument as in the proof of Lemma \ref{lem:unbounded} we see that there exists $w\in H^1_\mathrm{loc}(\rn)$ such that after passing to subsequences, $w_n\rh w$ weakly in $H^1_\mathrm{loc}(\rn)$, $w_n\to w$ in $L^p_\mathrm{loc}(\rn)$ and a.e. in $\rn$,  and $w\geq 0$. Since $u_n$ solves \eqref{eq:problem} with $\eps=\eps_n$, we obtain
$$
\langle w_n,\vp\rangle=\|u_n\|_{H^1(\o)}^{p-2}\irn Q_{\eps_n}|w_n|^{p-2}w_n\vp\qquad\text{for every \ }\vp\in\cC^\infty_c(\rn)
$$
and, as $\lim_{n\to\infty}\irn Q_{\eps_n}|w_n|^{p-2}w_n\vp=-\irn w^{p-1}\vp$ and $\|u_n\|_{H^1(\o)}\to \infty$, it follows that 
\[
\irn w^{p-1}\vp=0.
\]
Hence $w=0$.
\end{proof}
\medskip

\begin{proof}[Proof of Theorem \ref{thm:main_concentration}]
Let $U$ be an open bounded subset of $\rn$ such that $U\supset[-\frac{1}{2},\frac{1}{2}]^N$ and $\partial U\cap\z^N=\emptyset$. Without loss of generality, we assume that $\o\subset U$. Set $Z_1:=U\cap\z^N$ and $Z_2:=\z^N\smallsetminus Z_1$. Fix $0<\delta<\min\{\frac{1}{2},\,\dist(Z_1,\partial U)\}$ and an open bounded subset $V$ of $\rn$ such that $\overline U\subset V$ and $\dist(V,Z_2)>0$. Let $\psi\in\cC_c^\infty(\rn)$ be such that $0\leq\psi\leq 1$, $\psi(x)=1$ for $x\in U\smallsetminus B_\delta(Z_1)$ and $\psi(x) = 0$ for $x\in B_{\delta/2}(Z_1)$ and $x\in \rn\smallsetminus V$. Equality \eqref{eq:chi0} holds also for $\psi$. We re-write it as
\begin{equation} \label{ch}
\irn\psi(|\nabla u_n|^2+u_n^2) -\irn Q_{\eps_n}\psi|u_n|^p  = -\irn u_n\nabla u_n\cdot \nabla\psi .
\end{equation}
Let $\eps_n<\min\{\frac{\delta}{2},\,\dist(V,Z_2)\}$. Then, $\psi(x)=0$ for $x\in \t_{\eps_n}$ and therefore
$$\int_{U\smallsetminus B_\delta(Z_1)}|u_n|^p\leq\int_{\rn\smallsetminus \t_{\eps_n}}\psi|u_n|^p=-\irn Q_{\eps_n}\psi|u_n|^p.$$
Combining this inequality with \eqref{ch} we obtain
\begin{align*}
\int_{U\smallsetminus B_\delta(Z_1)}(|\nabla u_n|^2+u_n^2) + \int_{U\smallsetminus B_\delta(Z_1)}|u_n|^p 
& \le \irn\psi(|\nabla u_n|^2+u_n^2) -\irn Q_{\eps_n}\psi|u_n|^p\le C \int_V|u_n|\,|\nabla u_n|
\end{align*}
 for some $C>0$, and dividing by $\|u_n\|_{H^1(\Omega)}^2$ we get
\begin{equation*}
\int_{U\smallsetminus B_\delta(Z_1)}(|\nabla w_n|^2+w_n^2) + \|u_n\|_{H^1(\Omega)}^{p-2}\int_{U\smallsetminus B_\delta(Z_1)}|w_n|^p\leq C \int_V|w_n|\,|\nabla w_n|.
\end{equation*}
Since $V$ is covered by finitely many subsets of the form $y+\o$ with $y\in\z^N$ and $\|w_n\|_{H^1(y+\o)}\leq 1$, we have that $\|w_n\|_{H^1(V)}\leq C_1$ for some constant $C_1$ and,
as $w_n\to 0$ in $L^2(V)$, we obtain
\begin{equation} \label{equalzero} 
\lim_{n\to\infty}\Big(\int_{U\smallsetminus B_\delta(Z_1)}(|\nabla w_n|^2+w_n^2) + \|u_n\|_{H^1(\Omega)}^{p-2}\int_{U\smallsetminus B_\delta(Z_1)}|w_n|^p\Big) = 0.
\end{equation}
This immediately implies
\begin{align*}
\lim_{n\to\infty}\frac{\int_{U\smallsetminus B_\delta(Z_1)}(|\nabla u_n|^2+u_n^2)}{\int_U(|\nabla u_n|^2+u_n^2)}=\lim_{n\to\infty}\frac{\int_{U\smallsetminus B_\delta(Z_1)}(|\nabla w_n|^2+w_n^2)}{\int_U(|\nabla w_n|^2+w_n^2)}=0,
\end{align*}
which yields the first statement of Theorem \ref{thm:main_concentration}.

To prove the second one we first show that $\|u_n\|_{H^1(\o)}^{p-2}\int_U|w_n|^p$ is bounded away from $0$ . Let $\chi\in\cC_c^\infty(\rn)$ be such that $0\leq\chi\leq 1$, $\chi(x)=1$ for $x\in U$ and $\chi(x) = 0$ for $x\in \rn\smallsetminus V$. As in \eqref{ch}, we have
\[
\irn(\chi(|\nabla u_n|^2+u_n^2) -\irn Q_{\eps_n}\chi|u_n|^p  = -\irn u_n\nabla u_n\cdot \nabla\chi
\]
which for $\eps_n<\dist(V,Z_2)$ gives
\begin{align*}
 \int_U(|\nabla u_n|^2+u_n^2) - \int_U Q_{\eps_n}|u_n|^p \le \irn\chi(|\nabla u_n|^2+u_n^2) -\irn Q_{\eps_n}\chi|u_n|^p\le C \int_V|u_n|\,|\nabla u_n|.
\end{align*}
Dividing by $\|u_n\|_{H^1(\Omega)}^2$ and arguing as in \eqref{equalzero}, we then obtain
\[
\lim_{n\to\infty}\Big(\int_U(|\nabla w_n|^2+w_n^2) - \|u_n\|_{H^1(\Omega)}^{p-2}\int_U Q_{\eps_n}|w_n|^p\Big) = 0.
\]
Since $1=\io(|\nabla w_n|^2+w_n^2)\leq \int_U(|\nabla w_n|^2+w_n^2)$, it follows that 
\[
1\leq\lim_{n\to\infty}\|u_n\|_{H^1(\Omega)}^{p-2}\int_U Q_{\eps_n}|w_n|^p\leq\lim_{n\to\infty}\|u_n\|_{H^1(\Omega)}^{p-2}\int_U|w_n|^p. 
\]
As a consequence, using \eqref{equalzero} again, we get
\[
\lim_{n\to\infty}\frac{\int_{U\smallsetminus B_\delta(Z_1)}|u_n|^p}{\int_U|u_n|^p} = \lim_{n\to\infty}\frac{\|u_n\|_{H^1(\o)}^{p-2}\int_{U\smallsetminus B_\delta(Z_1)}|w_n|^p}{ \|u_n\|_{H^1(\o)}^{p-2}\int_U|w_n|^p} = 0. 
\]
from which the second statement of Theorem \ref{thm:main_concentration} follows.

The last statement follows from Lemma \ref{lem:unbounded}.
\end{proof}

\begin{remark}
\emph{
Note that in the special case when $U=\o$ Theorem \ref{thm:main_concentration} yields
\[
\lim_{n\to\infty}\frac{\int_{B_\delta(0)}(|\nabla u_n|^2+u_n^2)}{\int_{\Omega}(|\nabla u_n|^2+u_n^2)}=1 \qquad\text{and} \qquad\lim_{n\to\infty}\frac{\int_{B_\delta(0)}|u_n|^p}{\int_{\Omega} |u_n|^p}=1.
\]
Note also that by inspecting the proof of Theorem \ref{thm:main_concentration} we see that the above result still holds if $\Omega$ is replaced by $y+\o$ provided $y\in\zn\smallsetminus\{0\}$ is such that $\|u_n\|_{H^1(y+\Omega)}\to\infty$.  In this case we have 
\begin{equation} \label{aaa}
\lim_{n\to\infty}\frac{\int_{B_\delta(y)}(|\nabla u_n|^2+u_n^2)}{\int_{y+\Omega}(|\nabla u_n|^2+u_n^2)}=1 \qquad\text{and} \qquad\lim_{n\to\infty}\frac{\int_{B_\delta(y)}|u_n|^p}{\int_{y+\Omega} |u_n|^p}=1.
\end{equation}
More precisely,  we take $w_n := u_n/\|u_n\|_{H^1(y+\o)}$, and as in Lemma \ref{weaktozero} we show that $w_n\rh 0$ weakly in $H^1(y+\o)$ (the argument is in fact simpler because we do not need $w$ to be defined outside of $y+\o$). Next we choose an open set $V$ such that $\overline \o\subset V\subset (-\frac34,\frac34)^N$ and, assuming in addition that $\o$ has a regular  boundary ($\partial\o\in\cC^{0,1}$ suffices), we continuously extend  $w_n|_{y+\o}$ to $\widetilde w_n\in H^1(y+V)$ \cite[Theorem 7.25]{gt}. Now we can follow the argument of Theorem \ref{thm:main_concentration}. Thus \eqref{aaa} holds for \emph{all} regular $\o$ if $\|u_n\|_{H^1(y+\o) }\to\infty$. Using this it is easy to see that  it must hold for all $\o$ satisfying the assumptions at the beginning of this section. }

\emph{Observe that local concentration at more than one point does not exclude global concentration as in \eqref{eq:concentration_as}. It is unclear to us whether local concentration at a unique point yields global concentration. }
\end{remark}

\section{Open problems} \label{op}

As mentioned in the introduction, replacing $\t_\eps$ by a single ball $B_\eps(0)$ allows a more accurate description of the concentration behavior. Indeed, after proper rescaling, the solutions exhibit a non-trivial limiting profile \cite{fw} and the concentration in the $H^1$ and $L^p$ norms holds globally, as indicated by \eqref{eq:concentration_as}, see \cite{as}. Establishing the veracity of these facts for $\t_\eps$ seems to be rather delicate. We formulate some open questions.  

\begin{problem}
\emph{Let $v_n(x) := \eps_n^{\frac2{p-2}}u_n(\eps_nx)$ where $u_n$ is a positive least energy solution to problem \eqref{eq:problem} for $\eps=\eps_n$. Then $v_n$ satisfies the equation
\begin{equation} \label{rescaled}
-\Delta v_n + \eps^2v_n = \tilde Q_{\eps}(x)|v_n|^{p-2}v_n, \qquad v_n\in H^1(\rn),
\end{equation}
with $\eps=\eps_n$, where
\begin{equation*}
\tilde Q_\eps(x):=
\begin{cases}
1 & \text{if \ }x\in\tilde\t_\eps:=\bigcup\limits_{y\in\z^N}B_1(\frac y{\eps}),\\
-1 & \text{otherwise}.
\end{cases}
\end{equation*}
The corresponding functional is
\[
I_{\eps}(v) := \frac12\|v\|^2_\eps -\frac1p\irn \tilde Q_{\eps}|v|^p,\qquad\text{where \  \ }\|v\|^2_\eps:=\irn(|\nabla v|^2+\eps^2v^2),
\]
and $v_n$ is a least energy solution of the rescaled equation \eqref{rescaled}. Let $\cM_\eps$ be the Nehari manifold corresponding to \eqref{rescaled}. Fix $\vp\in\cC^\infty_c(B_1(0))$, $\vp\neq 0$. Then it is easy to see that $t_\eps\vp\in \cM_\eps$ if 
\[
t_\eps = \left(\frac{\|\vp\|_\eps^2}{\irn \tilde Q_\eps|\vp|^p}\right)^{\frac1{p-2}} = \left(\frac{\|\vp\|_\eps^2}{\irn|\vp|^p}\right)^{\frac1{p-2}},
\]
which gives (cf. \eqref{eq:energy_on_nehari})
$$
\frac{p-2}{2p}\|v_n\|^2_{\eps_n} =  \inf_{v\in\cM_{\eps_n}}I_{\eps_n}(v) \leq  I_{\eps_n}(t_{\eps_n}\vp) = \frac{p-2}{2p}\left(\frac{\|\vp\|^2_{\eps_n}}{\Big(\irn |\vp|^p\Big)^{2/p}}\right)^\frac{p}{p-2}.
$$
Since the right-hand side is bounded as $\eps_n\to 0$, the sequence $(v_n)$ is bounded in $D^{1,2}(\rn)$. So passing to a subsequence, $v_n\rh v$ weakly in $D^{1,2}(\rn)$ and $v_n\to v$ strongly in $L^p_\mathrm{loc}(\rn)$ and a.e. in $\rn$. Using this and testing \eqref{rescaled} with $\vp\in\cC_c^\infty(\rn)$ it is easy to see that $v$ is a solution to the equation
\[
-\Delta w = Q(x)|w|^{p-2}w,
\]
where
\begin{equation*}
Q(x):=
\begin{cases}
1 & \text{if \ }x\in B_1(0),\\
-1 & \text{otherwise}.
\end{cases}
\end{equation*}}

\emph{When $\t_\eps$ is replaced by $B_\eps(0)$ and $N\ge 3$ it is known that the weak limit $v$ is a positive least enegy solution for  this equation in $D^{1,2}(\rn)\cap L^p(\rn)$ and that it is a limiting profile for $v_n$, that is, $v_n\to v$ strongly in $D^{1,2}(\rn)\cap L^p(\rn)$; see \cite{fw,chs}.}

\emph{It would be interesting to find out whether this also holds true in our case, that is, whether the weak limit $v$ is positive or whether $v=0$. If $v$ is positive, is it true that $v_n\to v$ strongly in $D^{1,2}(\rn)\cap L^p(\rn)$?}
\end{problem}

\begin{problem}
\emph{Are the following global concentration statements
\begin{equation*}
\lim_{n\to\infty}\frac{\int_{B_\delta(0)}(|\nabla u_n|^2+u_n^2)}{\irn(|\nabla u_n|^2+u_n^2)}=1\qquad\text{and}\qquad\lim_{n\to\infty}\frac{\int_{B_\delta(0)}|u_n|^p}{\irn |u_n|^p}=1
\end{equation*}
true, after a suitable translation by elements of $\z^N$?}

\emph{We note that if instead of $\z^N$ we take a self-focusing core having a finite number of points, the corresponding statements are true, see \cite{as}}.
\end{problem}

\begin{problem}
\emph{Do we have $\lim_{n\to\infty}u_n=0$ in $L^p(\rn\smallsetminus B_\delta(0))$, or at least in $L^p(\Omega\smallsetminus B_\delta(0))$ for $N\ge 3$ and $p\in(\frac{2N-2}{N-2},2^*)$, as in \cite[Remark 6.4]{css} and \cite{as}?}
\end{problem}

\medskip

\noindent\textbf{Acknowledgements.}
M. Clapp thanks Stockholm University and A. Szulkin thanks Instituto de Matemáti\-cas Unidad Juriquilla UNAM for the kind hospitality.  
A. Saldaña was supported by UNAM-DGAPA-PAPIIT (Mexico) grant IN102925  and by CONAHCYT (Mexico) grant CBF2023-2024-116. M. Clapp and  A. Szulkin were supported in part by a grant from the Magnuson foundation of the Royal Swedish Academy of Sciences.

\bigskip

\begin{flushleft}
\textbf{Mónica Clapp}\\
Instituto de Matemáticas\\
Universidad Nacional Autónoma de México \\
Campus Juriquilla\\
76230 Querétaro, Qro., Mexico\\
\texttt{monica.clapp@im.unam.mx} 

\bigskip 

\textbf{Alberto Saldaña}\\
Instituto de Matemáticas\\
Universidad Nacional Autónoma de México \\
Circuito Exterior, Ciudad Universitaria\\
04510 Coyoacán, Ciudad de México, Mexico\\ \texttt{alberto.saldana@im.unam.mx}

\bigskip

\textbf{Andrzej Szulkin}\\
Department of Mathematics\\
Stockholm University\\
106 91 Stockholm, Sweden\\
\texttt{andrzejs@math.su.se}
\end{flushleft}

\end{document}